\documentclass[a4paper]{amsart}
\usepackage[utf8]{inputenc}
\usepackage{amssymb,amsthm,amsmath}
\usepackage{graphicx}
\usepackage{float}
\usepackage{morefloats} % allows for better placement of figures
\usepackage{fullpage}
\usepackage[colorlinks,linkcolor=blue,urlcolor=blue,citecolor=blue,destlabel=true]{hyperref}
\usepackage{color}
\usepackage{svg}

\newtheoremstyle{myremark} % name
    {7pt}                    % Space above
    {7pt}                    % Space below
    {}  	                 % Body font
    {}                           % Indent amount
    {\bf}       	         % Theorem head font
    {.}                          % Punctuation after theorem head
    {.5em}                       % Space after theorem head
    {}  % Theorem head spec (can be left empty, meaning ???normal???)
    
\theoremstyle{plain}
\newtheorem{lemma}{Lemma}[section]

\newtheorem{corollary}[lemma]{Corollary}

\newtheorem{theorem}[lemma]{Theorem}
\theoremstyle{definition}
\newtheorem{conjecture}[lemma]{Conjecture}
\newtheorem{definition}[lemma]{Definition}
\newtheorem{question}[lemma]{Question}
\theoremstyle{myremark}
\newtheorem{remark}[lemma]{Remark}
\theoremstyle{definition}

\newcommand{\N}{\ensuremath{\mathbb{N}}}

\newcommand{\R}{\ensuremath{\mathbb{R}}}

\title{Lions and contamination, triangular grids, and Cheeger constants}
\author{Henry Adams}
\email{henry.adams@colostate.edu}
\author{Leah Gibson}
\email{leah.gibson@rams.colostate.edu }
\author{Jack Pfaffinger}
\email{jack.pfaffinger@colostate.edu}

\begin{document}

\maketitle

\begin{abstract}
Suppose each vertex of a graph is originally occupied by contamination, except for those vertices occupied by lions.
As the lions wander on the graph, they clear the contamination from each vertex they visit.
However, the contamination simultaneously spreads to any adjacent vertex not occupied by a lion.
How many lions are required in order to clear the graph of contamination?
We give a lower bound on the number of lions needed in terms of the Cheeger constant of the graph.
Furthermore, the lion and contamination problem has been studied in detail on square grid graphs by Brass et al.\ and Berger et al., and we extend this analysis to the setting of triangular grid graphs.
\end{abstract}

\section{Introduction}
In the ``lions and contamination'' pursuit-evasion problem~\cite{BGGK,brass2007escaping,dumitrescu2008offline}, lions are tasked with clearing a square grid graph consisting of vertices and edges.
At the start of the problem, all vertices occupied by lions are considered cleared of contamination, and the rest of the vertices are contaminated.
In each time step, the lions move along the edges of the grid, and each new vertex they occupy becomes cleared.
However, the contamination can also travel along the edges of the grid not blocked by a lion and re-contaminate previously cleared vertices.
How many lions are needed to clear the grid, independent of the starting position of the lions?

Certainly $n$ lions can clear an $n\times n$ grid graph by sweeping from one side to the other.
One might conjecture that $n$ lions are required to clear an $n\times n$ grid graph.
However, in general it is not yet known whether $n-1$ lions suffice or not.
As a lower bound, the paper~\cite{brass2007escaping} proves that at least $\lfloor \frac{n}{2} \rfloor + 1$ lions are required to clear an $n\times n$ grid graph.
The details of the discretization certainly matter, in the following sense.
For a $n\times n\times n$ grid graph, one might expect that $n^2$ lions are required, but~\cite{BGGK} shows that when $n=3$, only $8=n^2-1$ lions suffice to clear a $3\times 3\times 3$ grid.

We consider the case of planar triangular grid graphs, under various models of lion motion.
%When the underlying grid is triangular, intuitively lions need to move one-at-a-time in order to perform a sweep.
Given a strip discretized by a triangular grid graph, $n$ lions can clear a strip of height $n$ when they are allowed to move one-at-a-time (this is allowed in the standard model of lion motion).
However, we conjecture that $n$ lions do not suffice when all lions must move simultaneously.
In the setting of simultaneous motion (which we refer to as ``caffeinated lions,'' as the lions never take a break), we show that $\lfloor \frac{3n}{2} \rfloor$ caffeinated lions suffice to clear a strip of height $n$.
Furthermore, via a comparison with~\cite{BGGK,brass2007escaping}, we show that $\lfloor \frac{n}{2} \rfloor$ lions are insufficient to clear a triangulated rhombus, in which each side of the rhombus has length $n$.
Lastly, for an equilateral triangle discretized into smaller triangles, with $n$ vertices per side, we conjecture that $\frac{n}{2\sqrt{2}}$ lions are not sufficient to clear the graph.

Furthermore, in the setting of an arbitrary graph $G$, we give a lower bound on the number of lions needed to clear the graph in terms of the \emph{Cheeger constant} of the graph.
The Cheeger constant, roughly speaking, is a measure of how hard it is to disconnect the graph into two pieces of approximately equal size by cutting edges~\cite{Chung_1997}.
The use of the Cheeger constant in graph theory is inspired by its successful applications in Riemannian geometry~\cite{cheeger1969lower}.
Our bound on the number of lions in terms of a graph's Cheeger constant is quite general (it holds for any graph), and therefore we do not expect it to be sharp for any particular graph.

We give background definitions and notation in Section~\ref{sec:background}, study triangular grid graphs in Section~\ref{sec:triangular}, and explain the connection to the Cheeger constant in Section~\ref{sec:cheeger}.
We ask many open questions in Section~\ref{sec:conclusion} that we hope will inspire future work.

\section{Related work}

The lion and contamination problem considered in this paper is only one of many pursuit-evasion problems that occur in graphs.
See~\cite{parsons1978pursuit} for an early treatment of an evasion problem on graphs, and see~\cite{fomin2008annotated} for a bibliography of related problems and papers.
The graph-clear problem explored by~\cite{kolling_carpin_2007} is a model for surveillance tasks and can be used to help determine the number of robots needed to patrol a large enclosed area.
In this problem, the contamination that must be cleared lives on the edges of the graph, whereas in the lion and contamination problem we consider the contamination lives on the vertices.
Pursuit-evasion problems have applications to air traffic control~\cite{bacsar1998dynamic}, robot motion planning~\cite{latombe2012robot,lavalle1998optimal}, defense~\cite{isaacs1999differential}, collision avoidance~\cite{fox1997dynamic}, and tracking~\cite{hajek2008pursuit}.

Perhaps the most developed graph-based pursuit-evasion challenge is the problem of Cops and Robbers.
This is a particular form of graph evasion problem which involves a set of cops moving along the vertices of a graph that attempt to capture a single robber also moving on the vertices of a graph.
In this version the cops and robbers alternate moving, rather than both moving at the same time.
The robber is caught if a cop is able to move onto the vertex the robber is occupying.
For a vertex $v$, define $N[v]$ to be the set containing $v$ and all vertices adjacent to $v$.
A vertex $v$ is called a $\it{corner}$ if there exists a vertex $u$ such that $N[v] \subseteq N[u]$.
A single cop is able to catch the robber on a finite graph if and only if some sequence of deleting corners results in a single vertex remaining.
This condition is known as the graph being \textit{dismantlable}.
For more information see~\cite{Bonato_Nowakowski_2011}.

There are many variants of pursuit-evasion problems that take place in Euclidean space instead of in graphs; see~\cite{chung2011search} for a survey.
The most famous such problem may be the lion and man problem proposed by Rado in 1925~\cite{littlewood1986littlewood}, in which a single man is chased by a single lion of the same speed.
The lion and man problem has been studied both with continuous time and with discrete time~\cite{alonso1992lion,sgall2001solution}.
In many pursuit-evasion problems, there are numerous sensors or intruders, for example as in~\cite{liu2005mobility,Liu2013dynamic,chin2009detection}.
Kalman filters can be used for efficient target tracking~\cite{La2009flocking, Olfati-Saber2006flocking, Su2016distributed}, and neural networks can obtain dynamic coverage while learning previously unknown domains~\cite{Qu2014finite}.

A related class of evasion problems take place in mobile sensor networks, as studied from the topological perspective in~\cite{EvasionPaths,Coordinate-free,de2007coverage}.
Here both space and time are continuous, and ball-shaped sensors wander in a bounded domain.
The perspective taken is that of minimal sensing, in that the sensors do not know their locations but instead only measure connectivity information.
Furthermore, the motion of the sensors is arbitrary --- perhaps the sensors are floating in the ocean or blown in the air by wind.
There is no speed bound for either the sensors or the intruders.
One would like to detect when the sensors have necessarily captured all possible intruders, but without using location information.
A slightly different perspective is taken in~\cite{adams2021efficient}, namely, what random models of sensor motions lead to faster mobile coverage?
The relation to the lion and contamination problem in this paper can be partially viewed as a passage from continuous to discrete.
If one approximates a continuous plane instead by a square or triangular grid graph, then the number of lions needed to clear that graph provides a discrete approximate analogue of the number of sensors that might be needed in a continuous model on the plane.

In contrast with the aforementioned models, in the lion and contamination pursuit evasion problem that we will consider (defined rigorously in the following section), both space and time are discrete, and the pursuers and evaders move simultaneously with bounded speed.

\section{Notation and definitions}
\label{sec:background}

We begin by providing notation and definitions for various models of lion motion, for how the contamination spreads, and for various types of triangular grid graphs.
For additional background information, we refer the reader to~\cite{BGGK,brass2007escaping}.

\subsection{Graphs}

In this paper we consider only finite simple graphs.
A finite simple graph $G=(V,E)$ contains a finite set of vertices $V$.
The set of (undirected) edges $E$ is a collection of 2-subsets from $V$.
We may also use the notation $V(G)$ or $E(G)$ to denote the vertex set of $G$ and edge set of $G$, respectively.
If $\{u,v\}$ is an edge in $E$, then we say that the vertices $u$ and $v$ are \emph{adjacent}, and we often denote this edge as $uv$.
Visually, we represent adjacency by drawing a line segment between the vertices $u$ and $v$.

An important concept for this paper will be the boundary of a set of vertices in a graph.

\begin{definition}
Let $G$ be a graph with vertex set $V$.
We define the \emph{boundary} of a vertex subset $S\subseteq V$, denoted $\partial S$, to be the collection of all vertices in $S$ that share an edge with some vertex of $V\setminus S$.
That is,
\[\partial S = \{v\in S~|~uv\in E(G) \text{ for some } u\notin S\}.\]
\end{definition}

\begin{definition}
A \emph{walk} $\pi$ in a graph $G$ is an ordered list $(\pi(1), \pi(2), \dots, )$ of vertices where each $\pi(t) \in V(G)$, and $\pi(t)$ is adjacent to $\pi(t+1)$ for all $t$.
\end{definition}
We note that a walk is allowed to retrace its steps --- i.e., the vertices visited need not be distinct.

\subsection{Lion Motion}

Each lion occupies a vertex of the graph.
In this evasion problem, time is discrete.
At each turn, a lion can either stay where it is, or move across an edge to an adjacent vertex.
Multiple lions are allowed to occupy the same vertex.

The main restrictions on lion motion that we consider are caffeinated lions and polite lions.
In the caffeinated model, all lions must move at every time step, and in the polite model, at most one lion can move at once.

\begin{definition}[Caffeinated Lions]
In the caffeinated model of lion motion, every lion must move at each turn.
In other words, between turns no lion is allowed to remain in place at its current location.
\end{definition}

\begin{definition}[Polite Lions]
In the polite model of lion motion, at most one lion is allowed to move at each turn.
All other lions must remain at their current vertex.
\end{definition}

When we simply say ``lions'' without specifying, we mean lions that are neither caffeinated nor polite --- any lion can move or stay put at any turn.

\subsection{Contamination motion}

Every vertex that is not originally occupied by a lion is \emph{contaminated}.
When a lion moves to a new vertex, that new vertex becomes \emph{cleared} of contamination.
A cleared vertex $v$ at time step $t$ becomes recontaminated at the next time step $t+1$ if 
\begin{itemize}
    \item it is not occupied by a lion at time step $t+1$, and
    \item if it is adjacent to a contaminated vertex $u$ at time $t$, and there is no lion that crosses the edge from $v$ to $u$ between times $t$ and $t+1$ (which would block the contamination from moving across this edge).
    Note that in the same time step that a lion leaves a vertex, it can become recontaminated.
\end{itemize}

We let $C(t)$ denote the set of cleared vertices at time $t$.
We say that the lions have \emph{cleared} or \emph{swept} the graph $G$ if at any time $t$, all of the vertices in the graph are cleared of contamination.

\begin{definition}[Sweep]
A sweep of a graph is the movement of the lions that results in the graph being completely cleared of contamination at time $t$, i.e.\ $C(t) = V(G)$.
\end{definition}
We are interested in finding the fewest number of lions required to clear a graph.

We note that if $G$ is a connected graph, and if $k$ lions can sweep $G$ from a certain starting position, then $k$ lions can sweep $G$ from \emph{any} starting position.
However, this is not necessarily the case if the lions are caffeinated---unless $G$ contains an odd cycle (see Section~\ref{sssec:starting}).

\subsection{Triangular grid graphs}

We begin by bounding the number of lions needed to clear triangular grid graphs, which we define now.

Let $T_n$ be the $n$th triangular number for $n \in \N$.
That is, $T_n = \frac{n(n+1)}{2}$.

\begin{definition}
[Triangular Grid]
Let $P_n$ be the planar graph which forms an equilateral triangle of side-length $n-1$ (with $n$ vertices on each side), subdivided into a grid of equilateral triangles as drawn in Figure~\ref{fig:P_5}.
\end{definition}

\begin{figure}
    %\sidecaption 
    \includegraphics[width = .3\textwidth]{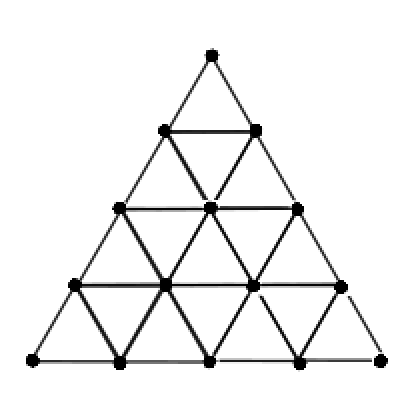}
    \caption{The triangular grid graph $P_5$.}
    \label{fig:P_5}
\end{figure}

Note that $P_n$ has $\frac{n(n+1)}{2} = T_n$ vertices and $\frac{3n(n-1)}{2}$ edges.
\footnote{Proof: The number of vertices are the $n$th triangular number defined by $\sum_{i = 0}^{n} i = \frac{n(n+1)}{2}$, since $T_i$ is constructed by taking $T_{i-1}$ and adding $i$ vertices.
Denote the number of edges in $T_n$ as $E_n$.
These numbers are determined by the recurrence relation $E_n = E_{n-1} + 3(n-1)$.
To see this, suppose that $T_{n-1}$ is drawn and place $n$ vertices below it.
Each of the $n-1$ vertices in row $n-1$ will have two edges drawn between itself and the vertices in the last row of $n$ vertices.
The last row of $n$ vertices will be connected together by another $n-1$ edges.
So the equivalence relation is established, and it satisfies $E_n = \frac{3n(n-1)}{2}$ for $n \geq 1$.}

\begin{definition}[Triangular Lattice]
Let $R_{n,l}$ be the planar graph which forms a parallelogram of height $n$ vertices and length $l$ vertices, subdivided into a grid of equilateral triangles as drawn below.
\end{definition}

\begin{figure}
    \centering
    \includegraphics[width = .5\textwidth]{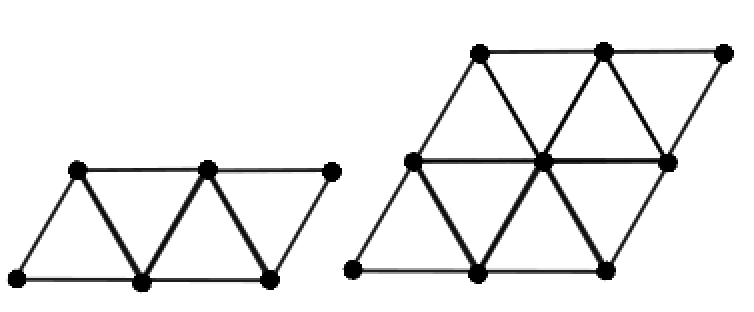}
    \caption{Example of $R_{2,3}$ and $R_{3,3}$ (from left to right).}
    \label{fig:cut_ribbon}
\end{figure}

\section{Lions and contamination on triangular grid graphs}
\label{sec:triangular}

We study the number of lions needed to clear triangular grid graphs, for various shapes of graph domain, and under various different models of lion motion.

\subsection{Sufficiency of $n$
lions on a triangulated strip}
\label{ssec:n-suffic}

\begin{theorem}
Let $n$ and $l$ be positive integers.
Then $n$ lions suffice to clear the grid $R_{n,l}$ of contamination.
\end{theorem}

\begin{proof}
There is a specific sweeping formation that we use for the $n$ lions to clear the grid of contamination, but since there is no restriction on lion movement and the lions live on a connected graph, all lions can move into this clearing position without issue.

\begin{figure}
    \centering
    \includegraphics[width=.8\textwidth]{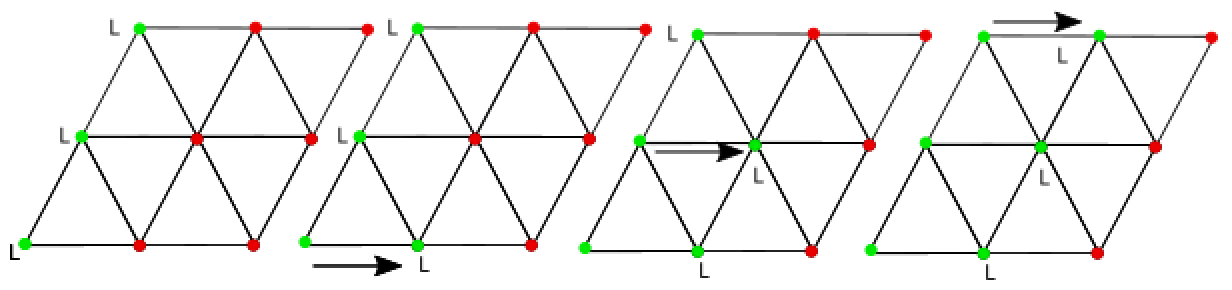}
    \caption{Sweeping formation on $R_{3,3}$.
    The red vertices are contaminated, the green vertices are cleared, and the locations of lions are marked with `L'.}
    \label{fig:R3_ex}
\end{figure}

In the graph $R_{n,l}$ we let the bottom row be labeled row 1, the next row up is row 2, and this continues to row $n$.
We can apply a sweeping method using $n$ lions which move until they are positioned on the leftmost vertices of the grid, i.e.\ the leftmost diagonal column.
We complete the sweep in the following manner as illustrated in Figure~\ref{fig:R3_ex}:
Let the lion on row 1 travel one step to the right.
The vertex it previously occupied is cleared and protected from future contamination in the next step.
Next, let the lion in row 2 travel one step to the right.
The vertex it previously occupied is cleared and protected from future contamination in the next step.
Now, we let the lion in row 3 move one step to the right.
Similarly, the previously occupied vertex is cleared and protected.
We repeat this for $n$ lions.
Once each lion has moved one step to the right, the previous column is cleared and we repeat the process again starting with the lion on row 1.
We continue this sweep until the entire length of the graph $R_{n,l}$ is cleared.
\end{proof}

We note that this clearing sequence is obtainable with polite lions, but not with caffeinated lions, as demonstrated in Figure~\ref{fig:caff_fail}.

\begin{figure}
    \centering
    \includegraphics[width = .5\textwidth]{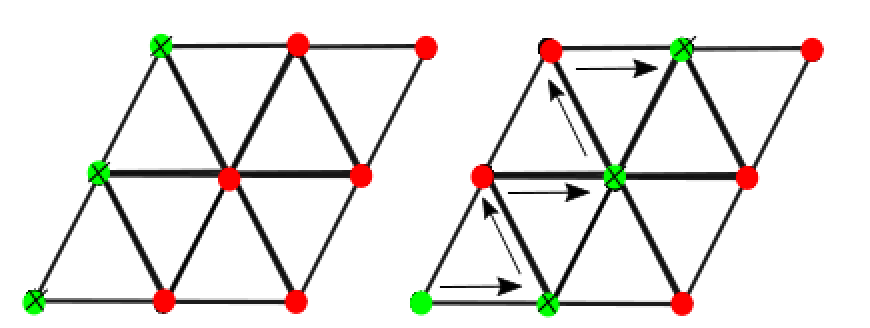}
    \caption{Under caffeinated lion motion, $n$ caffeinated lions cannot sweep $R_{n,l}$ by moving from left to right, as contamination can travel along the diagonal edges and recontaminate previously cleared vertices.
    The red vertices are contaminated, the green vertices are cleared, and the locations of caffeinated lions are marked with `X'.}
    \label{fig:caff_fail}
\end{figure}

\subsection{Sufficiency of $\lfloor\frac{3n}{2}\rfloor$ caffeinated
lions on a triangulated strip}

We now restrict lion movement so that the lions are caffeinated, i.e.\ each lion must move at every step.
We show that $\lfloor\frac{3n}{2}\rfloor$ caffeinated lions suffice to clear the triangular lattice $R_{n,l}$.
Indeed, we provide a set formation where the lions may start in order to clear the grid using only $\lfloor\frac{3n}{2}\rfloor$ caffeinated lions.
In order to prove that $\lfloor\frac{3n}{2}\rfloor$ \emph{caffeinated} lions are sufficient independent of their starting positions, we must also show that the caffeinated lions can move to this initial formation from any starting position.

\subsubsection{Sweeping formation for $\lfloor\frac{3n}{2}\rfloor$ caffeinated lions}
\label{ssec:sweeping-formation}

Consider the graph $R_{n,l}$.
We let the $\lfloor\frac{3n}{2}\rfloor$ lions line vertically such that the lions create a ``wall of triangles'' in which each row alternates having one or two lions, starting with one lion in the top row.
These lions form a vertical wall that stretches from the bottom to top of the grid; see to Figure~\ref{fig:caff_formation}(left) for this formation.
%Three adjacent vertices in this wall should form a small, triangular 3-cycle.
We will show in Lemma \ref{lem:starting} that given any initial starting position, the caffeinated lions can travel into this initial sweeping formation.
Once in this formation, the caffeinated lions can sweep first to the left and then to the right to clear the contamination.
When the lions reach a corner of the grid, the lions sharing horizontal edges with the still contaminated vertices will continue their horizontal sweep.
The remaining lions rotate in a three cycle (or swap positions with a lion on an adjacent vertex) to remain caffeinated while the other lions sweep.
Refer to Figure~\ref{fig:caff_formation} for an example.

\begin{figure}
    \centering
    \includegraphics[width = 1\textwidth]{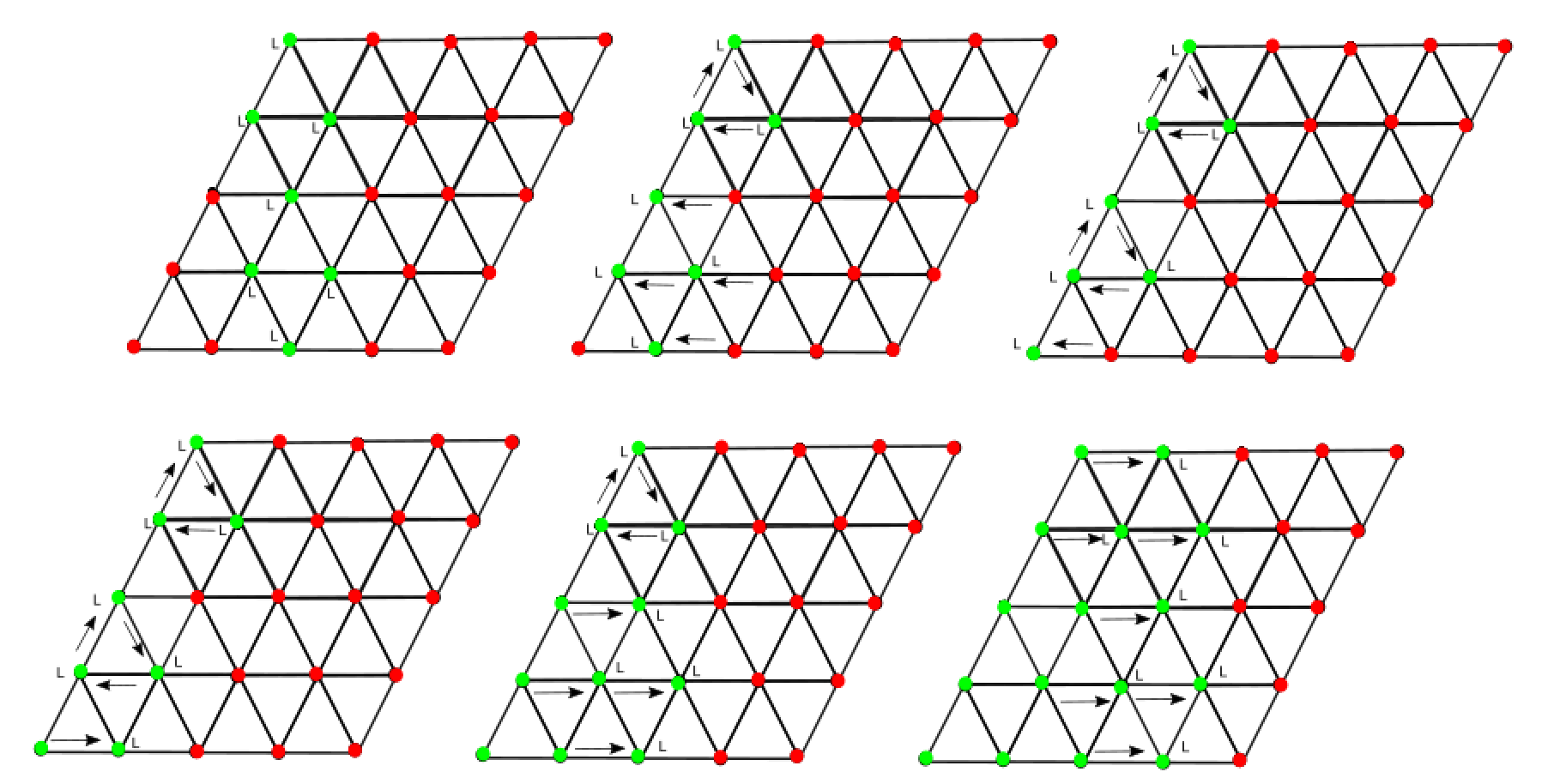}
    \caption{The sweeping procedure for caffeinated lions using a wall of triangles.}
    \label{fig:caff_formation}
\end{figure}

\subsubsection{Can lions get to some predetermined starting position?}
\label{sssec:starting}

We will use the following lemma to show that caffeinated moving lions may move to arbitrarily specified positions in the $R_{n,l}$ graph, no matter where their initial starting positions are.
This allows the caffeinated lions to move to the initial sweeping formation shown in Section~\ref{ssec:sweeping-formation}, regardless of their starting positions.

\begin{lemma}\label{lem:starting}
Consider any two vertices $u,v$ in the finite connected graph $R_{n,l}$.
Let $M$ be the length of the shortest path between $u$ and $v$.
Then for any $m\ge M$, there is a walk from $u$ to $v$ of length exactly $m$.
\end{lemma}

\begin{proof}
We are given the graph $R_{n,l}$; see Figure~\ref{fig:caff_walks}.
We label two arbitrary distinct vertices $u\neq v$ on this grid, where $u$ denotes the starting point of a lion, and $v$ is the desired ending location of the lion.

We will proceed by induction on $m$.
The case $m=M$ is clear since $M$ is defined as the length of the shortest walk from $u$ to $v$ in this finite connected graph.

For the inductive step, suppose there is a walk of length $m$ from $u$ to $v$; we claim there exists a walk of length $m+1$ from $u$ to $v$.
Consider the walk of length $m$.
Let the lion travel along this walk until it has taken $m-1$ steps.
At this moment, the lion will be on a vertex adjacent to $v$; call this vertex $q$.
Since every two adjacent vertices in the graph $R_{n,l}$ are part of a common 3-cycle, there is some vertex $s$ adjacent to both $q$ and $v$.
In the next step, let the lion move from vertex $q$ to $s$.
Once on vertex $s$, the lion has traveled a walk of length $p$, and then moves one more step to $v$ in a walk of length $m+1$.

By induction, it is possible to travel from $u$ to $v$ in a walk of length exactly $m$ steps for any $m\ge M$.
\end{proof}

\begin{figure}
    \centering
    \includegraphics[width=.7\textwidth]{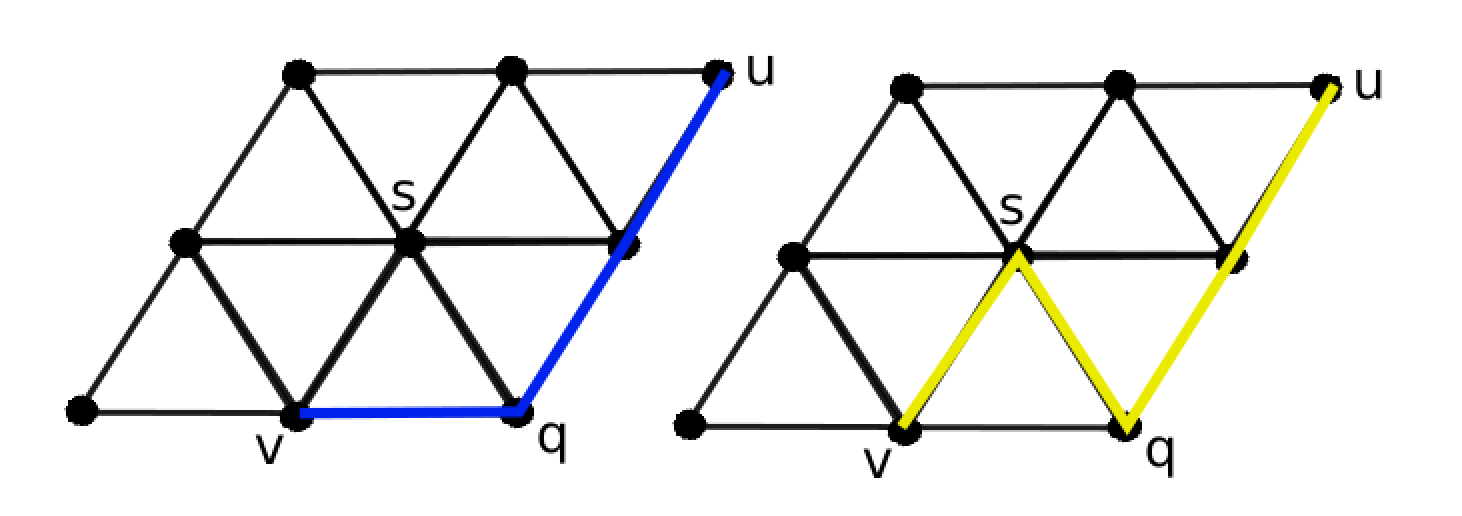}
    \caption{Figure accompanying the proof of Lemma~\ref{lem:starting}.}
    \label{fig:caff_walks}
\end{figure}

\begin{remark}
Lemma~\ref{lem:starting} more generally holds true for any connected graph $G$ containing an odd cycle.
\end{remark}

The property in Lemma~\ref{lem:starting} does not hold on the chessboard graph, i.e.\ a $n\times n$ square grid graph (see Figure~\ref{fig:s_n}).
This is because given any two vertices $u$ and $v$ in this graph, there is a parity (even or odd) such that any walk between $u$ and $v$ necessarily has length of that specified parity.
The same is true for any 2-colorable (i.e.\ bipartite) graph.

\begin{corollary}\label{cor:all_lions}
Consider the graph $R_{n,l}$.
Given any $k$ caffeinated lion starting positions $u_1,\ldots,u_k$ and any $k$ specified ending positions $v_1,\ldots,v_k$, there exists some $M\in \N$ so that we can have all lions move to arrive at the specified ending positions at (exactly) time step $M$.
\end{corollary}

\begin{proof}
Let $1\le i\le k$ be arbitrary.
By Lemma~\ref{lem:starting}, we know that there exists an integer $M_i\in \N$ such that for any vertices $u_i,v_i\in R_{n,l}$, and any $m\ge M_i$, there exists a walk of length exactly $m$ between $u_i$ and $v_i$.
Now, let $M=\max\{M_i~|~1\le i\le k\}$.
For any $m\ge M$, and for all $1\le i\le k$, there exists a walk of length exactly $m$ from $u_i$ to $v_i$.
Hence the $k$ caffeinated lions can simultaneously move from their initial locations to their desired positions in exactly $m$ simultaneous steps.
\end{proof}

\begin{corollary}
$\lfloor\frac{3n}{2}\rfloor$ caffeinated lions suffice to clear the grid $R_{n,l}$, no matter their starting locations.
\end{corollary}

\begin{proof}
It follows from Corollary~\ref{cor:all_lions} that the caffeinated lions can go into the sweeping formation from any starting position.
We have already shown in Section~\ref{ssec:sweeping-formation} that the caffeinated lions can clear the graph from here.
\end{proof}

\subsection{Insufficiency of $\lfloor\frac{n}{2}\rfloor$ lions on a triangulated square}
\label{ssec:insuffic-Rn}

We will use some of the methods and proofs from~\cite{BGGK} to show that $\lfloor\frac{n}{2}\rfloor$ (non-caffeinated) lions cannot clear $R_n:=R_{n,n}$.
In this proof, we stretch the ``rhombus'' graph $R_n$ to instead be drawn as a square triangulated by right triangles.
It should be noted that this grid holds all of the same properties as before (the isomorphism type of the graph is unchanged).

We define $S_n$ to be the $n \times n$ square grid graph with $n$ vertices per side as discussed in~\cite{BGGK}; see Figure~\ref{fig:s_n}.
When each square is subdivided via a diagonal edge, drawn from the top left to the bottom right, then we obtain a graph isomorphic to $R_n$ as shown in Figure~\ref{fig:r_n}.

\begin{figure}
\centering
\begin{minipage}{.5\textwidth}
  \centering
  \includegraphics[width=.4\linewidth]{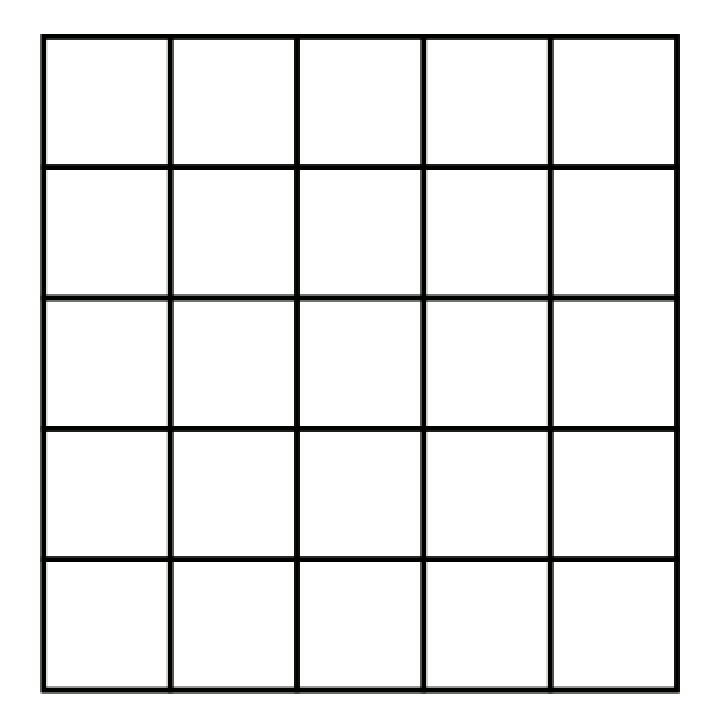}
  \caption{The graph $S_6$.}
  \label{fig:s_n}
\end{minipage}%
\begin{minipage}{.5\textwidth}
  \centering
  \includegraphics[width=.4\linewidth]{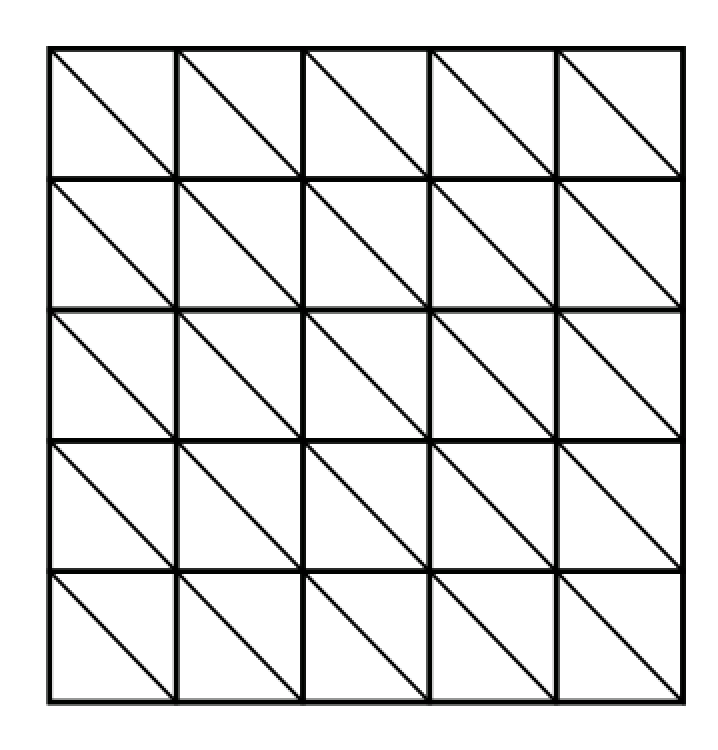}
  \caption{The graph $R_6$ drawn as a square instead of a rhombus.}
  \label{fig:r_n}
\end{minipage}
\end{figure}

The following two lemmas from~\cite{BGGK} hold in an arbitrary graph.

\begin{lemma}\label{lem:BGGK1}
(Lemma~1 of~\cite{BGGK})
Let $k$ be the number of lions on a graph.
The number of cleared vertices cannot increase by more than $k$ in one time step.
\end{lemma}

\begin{lemma}\label{lem:BGGK2}
(Lemma~2 of~\cite{BGGK})
Let $k$ be the number of lions.
If there are at least $2k$ boundary vertices in the set $C(t)$ of cleared vertices, then the number of cleared vertices cannot increase in the following step:
$|\partial C(t)| \ge 2k$ implies $|C(t+1)| \le |C(t)|$.
\end{lemma}

For the specific case of square grid graphs $S_n$,~\cite{BGGK} defines a ``fall-down transformation''.
This transformation $T$ takes any subset of the vertices of $S_n$, and maps it to a (potentially different) subset of the same size.
The first step in the fall-down transformation moves the subset of vertices down each column as far as possible.
The number of vertices in any column remains unchanged by this step.
The second step in the fall-down transformation is to then move each vertex as far as possible to the left-hand side of its row, maintaining the number of vertices in any row.

In the case of a square grid $S_n$,~\cite{BGGK} proves that the fall-down transformation does not increase the number of boundary vertices in a set:

\begin{lemma}\label{lem:BGGK4}
(Lemma~4 of~\cite{BGGK}) In the graph $S_n$, the fall-down transformation $T$ is monotone, meaning that the number of boundary vertices in a subset $S$ of vertices from $S_n$ does not increase upon applying the fall-down transformation.
\end{lemma}

Since the vertex set of $S_n$ is the same as the vertex set of $R_n$, we immediately get a fall-down transformation for $T$ acting on $R_n$ that maps a subset of vertices in $R_n$ to a subset of vertices in $R_n$.
We will show that this new fall-down transformation has the same monotonicity property, which is not a priori clear as the boundary of a set of vertices in $S_n$ might be smaller than the boundary of that same set of vertices in $R_n$.
Note that in general, if $v$ is a boundary vertex of a set $S$ in $S_n$, then $v$ is a boundary vertex of $S$ in $R_n$.
This implies that a set of vertices in $S_n$ has no more boundary vertices than if that set of vertices were in $R_n$.

Another comment is that when defining the fall down transformation $T$ on $S_n$, the choice of down vs.\ up or of left vs.\ right does not matter.
But these choices do matter when defining a fall-down transformation on $R_n$, due to the diagonal edges as drawn in Figure~\ref{fig:r_n}.
We want to emphasize we have chosen to map down and to the left; the following lemma in part depends on this choice.
Refer to Figure~\ref{fig:fdt_fail} to see that Lemma \ref{lem:FDT} fails if we instead allow $T$ to move down and to the right.

\begin{figure}
\centering
\includegraphics[width = .5\textwidth]{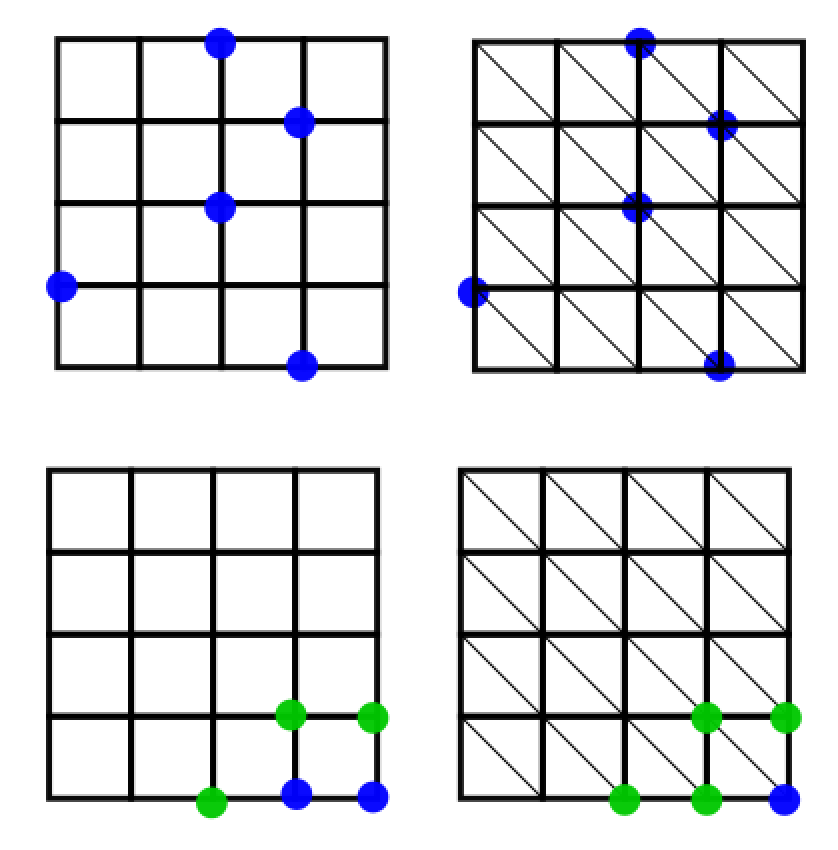}
\caption{Given the initial starting position of a set $S$ in blue on $S_n$ and $R_n$, notice that applying a modified fall-down transformation down and to the \emph{right} results in 3 boundary vertices in $S_n$ and 4 boundary vertices in $R_n$, shown in green.}
\label{fig:fdt_fail}
\end{figure}

\begin{lemma}\label{lem:FDT}
Let $S$ be a set of vertices in $R_n$ (or equivalently, in $S_n$).
The set of boundary vertices of $T(S)$ in $R_n$ is the same as the set of boundary vertices of $T(S)$ in $S_n$.
\end{lemma}

\begin{proof}
We will show that the set of boundary vertices of $T(S)$ in $R_n$ is the same as the set of boundary vertices of $T(S)$ in $S_n$.
First suppose that a vertex $v$ of $T(S)$ is a boundary vertex in $R_n$.
Referring to the diagram in Figure~\ref{fig:r_n-boundary}, this implies that one of $d, c, b$ or $a$ is not in $T(S)$.
However, we can reduce this to the case that one of $c$ or $b$ is not in $T(S)$, since $d \notin T(S) \Rightarrow c \notin T(S)$, and since $a \notin T(S) \Rightarrow b \notin T(S)$.
If $c \notin T(S)$, then $v$ is a boundary vertex of $T(S)$ in both $R_n$ and $S_n$.
The same is true if $b \notin T(S)$.
%Finally, by definition of the fall down transformation, if $a \notin T(S),$ then $b \notin T(S)$, so once again this implies that $v$ is a boundary vertex of $T(S)$ in both $R_n$ and $S_n$.
This proves that if $v$ is a boundary vertex of $T(S)$ in $R_n$, then it is a boundary vertex of $T(S)$ in $S_n$.
On the other hand, referring to Figure~\ref{fig:s_n-boundary}, if $v$ is a boundary vertex of $T(S)$ in $S_n$, then one of $c$ or $b$ is not in $T(S)$, which implies that $v$ is a boundary vertex of $T(S)$ in $R_n$ as well.
Therefore the set of boundary vertices of $T(S)$ in $R_n$ is the same as the set of boundary vertices of $T(S)$ in $S_n$.

\end{proof}

\begin{figure}
\centering
\begin{minipage}{.5\textwidth}
  \centering
  \includegraphics[width=.4\linewidth]{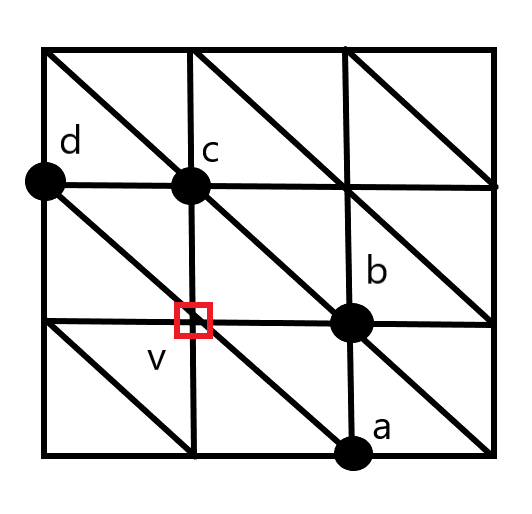}
  \caption{$R_n$.}
  \label{fig:r_n-boundary}
\end{minipage}%
\begin{minipage}{.5\textwidth}
  \centering
  \includegraphics[width=.4\linewidth]{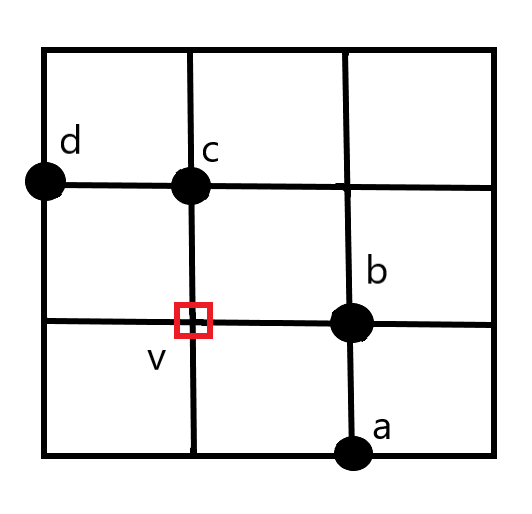}
  \caption{$S_n$.}
  \label{fig:s_n-boundary}
\end{minipage}
\end{figure}

We know that $T$ is monotone on $S_n$, i.e.\ that the number of boundary vertices does not increase as a result of applying $T$.
Lemma~\ref{lem:FDT} therefore immediately implies that $T$ is also monotone on $R_n$, i.e.\ that the number of boundary vertices of a set $S$ in $R_n$ is no more than the number of boundary vertices of $T(S)$ in $R_n$.

\begin{corollary}
In the graph $R_n$, the fall-down transformation $T$ is monotone, meaning that the number of boundary vertices in a subset $S$ of vertices from $R_n$ does not increase upon applying the fall-down transformation.
\end{corollary}

This now brings us to the last lemma.

\begin{lemma}\label{lem:BGGK5}
(Lemma~5 of~\cite{BGGK})
Any vertex set $S$ on $R_n$ satisfying $\frac{n^2}{2} - \frac{n}{2} < |S| < \frac{n^2}{2} + \frac{n}{2}$ has at least $n$ boundary vertices.
\end{lemma}

\begin{proof}
By Lemma~\ref{lem:BGGK4}, we know that the fall-down transformation $T$ acts monotonically on the number of boundary vertices in $R_n$.
Thus the proof of Lemma~5 from~\cite{BGGK} will also hold for our grid graph with diagonal edges added.

\end{proof}

\begin{theorem}
\label{thm:Dn}
$\lfloor\frac{n}{2}\rfloor$ lions do not suffice to clear $R_n$.
\end{theorem}

\begin{proof}
Suppose for a contradiction that $k \le \lfloor\frac{n}{2}\rfloor$ lions can clear $R_n$.
The lions will eventually have to clear all $n^2$ vertices.
By Lemma~\ref{lem:BGGK1}, we know that $|C(t+1)| - |C(t)| \le k \le \frac{n}{2}$ for all times $t$.
Thus there must be a time $t$ where $\frac{n^2}{2} - \frac{n}{4} \le |C(t)| \le \frac{n^2}{2} + \frac{n}{4}$ and $|C(t+1) > |C(t)|$.
But by Lemma~\ref{lem:BGGK5}, there are at least $n\ge 2k$ boundary vertices of $C(t)$ at time $t$, and so Lemma~\ref{lem:BGGK2} tells us that $|C(t+1)| \le |C(t)|$, which is a contradiction.
Thus $k \le \lfloor\frac{n}{2}\rfloor$ lions do not suffice to clear $R_n$.
\end{proof}

\subsection{Conjectured insufficiency of $\frac{n}{2\sqrt{2}}$ lions on a triangle}

In this section we consider the triangular grid graph $P_n$ with $n$ vertices on each of its three sides.
We conjecture that $\lfloor\frac{n}{2\sqrt{2}}\rfloor$ lions are not capable of clearing the triangular graph $P_n$ --- though we give only an incomplete possible attempted proof strategy/outline.
By contrast, the square graphs $S_n$ and $R_n$ with $n$ vertices per each of their four sides require at least $n/2$ lions to clear.
We ask in Question~7 in Section~\ref{sec:conclusion} how many lions are necessary to clear the triangular graph $P_n$ or the square graphs $S_n$ and $R_n$.

We use the notation $\alpha\approx\beta$ to mean that there is some small constant $\kappa$ such that $|\alpha-\beta| \le \kappa$.

\begin{conjecture}\label{conj:Pn}
It is not possible to clear $P_n$ with fewer than
$\approx\lfloor\frac{n}{2\sqrt{2}}\rfloor$ lions.
\end{conjecture}

The bound in this first conjecture is derived from later conjectures, which would provide a kind of ``isoperimetric inequality'' for sets $C$ in $P_n$.
We will show how Conjecture~\ref{conj:Pn} would be a consequence of the isoperimetric inequalities in the following Conjectures~\ref{con:pw_iso} and~\ref{con:bounds}.
We must first establish some new notation.

There are a number of ways that vertices in a set $C$ can be packed into the graph $P_n$ to have relatively few boundary vertices.
For the following definitions, we orient $P_n$ as drawn in Figures~\ref{fig:row} and~\ref{fig:icecream}.
A \textit{row packing} occurs when the vertices of $C$ fall so that each row of $P_n$ is completely filled before the row above it contains any vertices in $C$, and within each row the vertices are filled left to right.
The orientation for this packing is shown in Figure~\ref{fig:row}.
We define an \textit{ice cream come packing} to be a packing where the diagonal rows beginning in the lower left corner are filled one at a time, with the lowest vertex in each diagonal being filled first, and the next diagonal row cannot contain vertices in $C$ unless the diagonal row below it is already filled.
The orientation for this packing is shown in Figure~\ref{fig:icecream}.

\begin{figure}
    \centering
    \includegraphics[width = .3\textwidth]{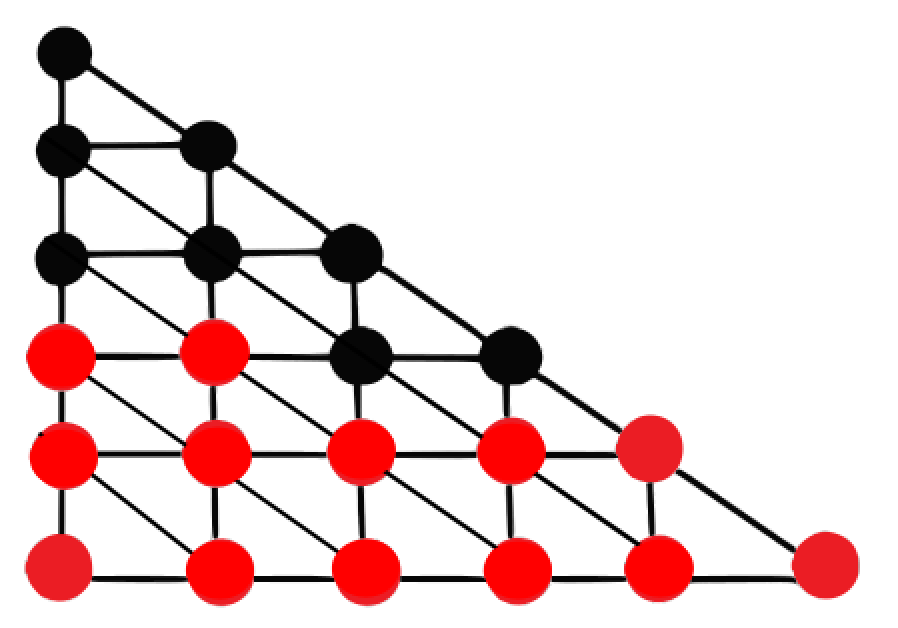}
    \caption{A row packing on $P_6$ with 13 vertices.}
    \label{fig:row}
\end{figure}

\begin{figure}
    \centering
    \includegraphics[width = .3\textwidth]{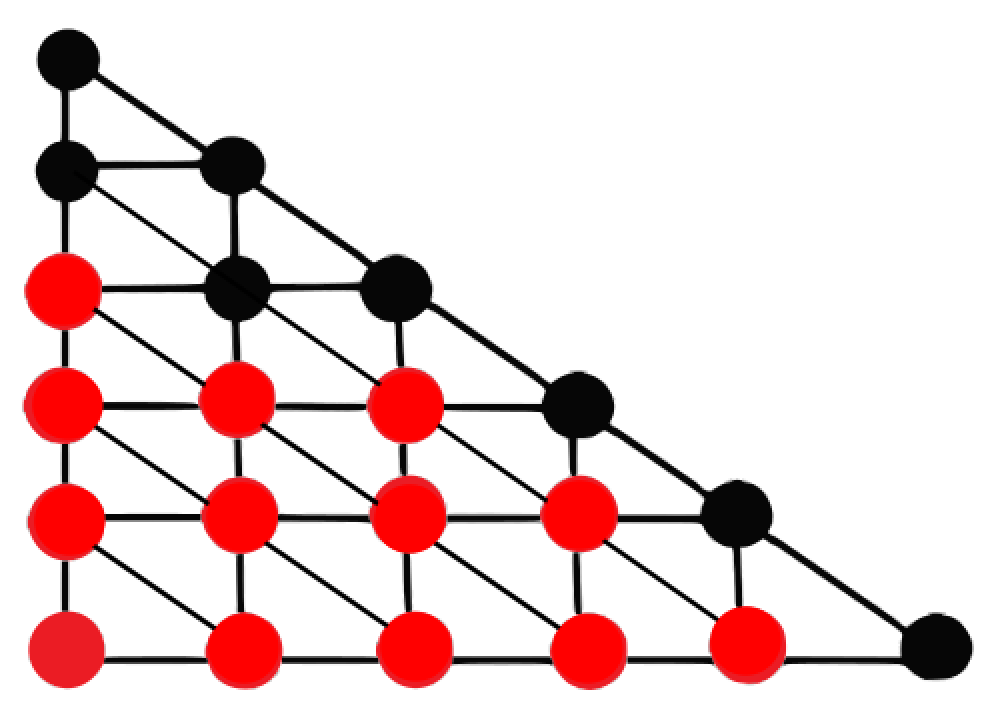}
    \caption{An ice cream cone packing on $P_6$ with 13 vertices.}
    \label{fig:icecream}
\end{figure}

The following two conjectures, if true, would give a type of isoperimetric inequality in the graph $P_n$.
%We write $f(n)\lesssim g(n)$ if there is some small constant $\kappa$ such that $f(n)\le g(n)+\kappa$ for all possible inputs $n$.
The next conjecture describes what we believe are the optimal ways to place vertices in $P_n$ while minimizing the number of boundary vertices.

\begin{conjecture}
\label{con:pw_iso}
If $C$ is a set of vertices in $P_n$, then $|\partial C|$ is at least as large as the minimum of 
\begin{itemize}
\item the number of boundary vertices in a row packing with $|C|$ vertices, and
\item the number of boundary vertices in an ice cream cone packing with $|C|$ vertices.
%\item  $|C| \lesssim \frac{T_n}{2}$, then $|\partial C|$ is greater than or equal to the number of boundary vertices for a set of the same size in an ice cream cone packing.
%\item $|C| \gtrsim \frac{T_n}{2}$, then $|\partial C|$ is greater than or equal to the number of boundary vertices for a set of the same size in a row packing.
\end{itemize}
\end{conjecture}

\begin{conjecture}
\label{con:bounds}
If $C$ is a set of vertices in $P_n$ with $|C|\approx T_{\lfloor\sqrt{T_n}\rfloor}$, 
then there are at least $\approx \lfloor \frac{n}{\sqrt{2}} \rfloor$ boundary vertices in $C$.

\end{conjecture}

\textbf{\emph{Possible proof sketch of Conjecture~\ref{con:bounds}, using Conjecture~\ref{con:pw_iso}}}

As we add vertices to an ice cream cone packing, the number of boundary vertices is nondecreasing (while the last diagonal row is still empty).
For $j\le n-1$, the number of boundary vertices in an ice cream cone packing of $\approx T_j$ vertices is $\approx j$.
By contrast, in a row packing, after the bottom row is full, the number of boundary vertices is nonincreasing as we add more vertices.
For $j\le n-1$, the number of boundary vertices in a row packing of $\approx T_n-T_j$ vertices is $\approx j$.
Due to these monotonicity properties, Conjecture~\ref{con:pw_iso} would imply that when the number of vertices we place is $T_j\approx T_n-T_j$, then we will have at least $\approx j$ boundary vertices.
We solve for $j$ in the equation $T_j\approx T_n-T_j$ to obtain $2T_j \approx T_n$, i.e.\ $j(j+1) \approx T_n$, giving $j^2 + j - T_n \approx 0$.
We solve to find $j \approx \frac{-1 \pm \sqrt{1 + 4T_n}}{2} \approx \frac{\sqrt{4T_n}}{2}=\sqrt{T_n}$.
Hence if $C$ is a set of vertices in $P_n$ with $|C|\approx T_{\lfloor\sqrt{T_n}\rfloor}$, 
then there are at least $j \approx \sqrt{T_n} \ge \lfloor \frac{n}{\sqrt{2}} \rfloor$ boundary vertices in $C$, where in the last inequality we have used the bound $T_n=\frac{n(n+1)}{2} \ge \frac{n^2}{2}$.
\vskip3mm

We can now show how Conjecture~\ref{conj:Pn} would follow from the above two (unproven) conjectures.

\textbf{\emph{Possible proof of Conjecture~\ref{conj:Pn} from Conjectures~\ref{con:pw_iso} and~\ref{con:bounds}}}
Suppose for a contradiction that $k$ is at most $\approx\lfloor\frac{n}{2\sqrt{2}}\rfloor$, and that $k$ lions can clear $P_n$.
Thus at the final time step, there must be $T_n$ cleared vertices.
Lemma \ref{lem:BGGK1} tells us that $|C(t+1)| - |C(t)| \le k$ for all times $t$.
There must have been a time $t$ such that the number of vertices in $C(t)$ was approximately $T_{\lfloor\sqrt{T_n}\rfloor}$ and $|C(t+1)| > |C(t)|$.
At that same time, the number of boundary vertices $|\partial C(t)|$ would have to be at least $\approx\lfloor\frac{n}{\sqrt{2}}\rfloor$ by Conjecture~\ref{con:bounds}.
By Lemma~\ref{lem:BGGK2}, since $|\partial C(t)|$ is at least $2k$, the number of boundary vertices cannot increase in the next time step, so in fact $|C(t+1)| \le |C(t)|$.
This is a contradiction, and thus fewer than $\approx \lfloor\tfrac{n}{2\sqrt{2}}\rfloor$ lions cannot clear $P_n$.

\section{Connection to Cheeger constant}\label{sec:cheeger}

In graph theory, the term \emph{Cheeger constant} refers a  numerical measure of how much of a bottleneck a graph has.
The term arises from a related quantity, also known as a Cheeger constant, that is used in differential geometry: the Cheeger constant of a Riemannian manifold depends on the minimal area of a hypersurface that is required to divide the manifold into two pieces~\cite{cheeger1969lower}.
For both graphs and manifolds, the Cheeger constant can be used to provide lower bounds on the eigenvalues of the Laplacian (of the graph or of the manifold).
In the graph theory literature, there are several different quantities known as the Cheeger constant, and they are each defined slightly differently.
For some more background and applications of Cheeger constants to graphs, see~\cite[Chapter~2]{Chung_1997}.
% \url{http://www.math.ucsd.edu/~fan/research/cb/ch2.pdf}
In this section, we show a relationship between the Cheeger constant of a graph and the number of lions needed to clear this graph of contamination.

%We will now define a new value for our graph, which we will call the \emph{Cheeger constant}.

\begin{definition}
For a graph $G$, let the Cheeger constant be
\[
    g := \min \left\{\frac{|\partial S|}{\min \{|S|,|\overline{S}| \}} : S \subseteq V(G),\ S \neq \emptyset,\ S \neq V(G) \right\},
\]
where $\partial S = \{v\in S~|~uv\in E(G),\ u\notin S\}$ (i.e.\ the set of boundary vertices), and $\overline{S}$ is the set $V(G) \setminus S$.
\end{definition}
This is the same constant that is studied in~\cite{bobkov2000lambda}, where it is called the \emph{vertex isoperimetric number} $h_{in}$.
% Clearly ours is <= theirs.
% If ours was >, then this would require a set S of size more than half the size of |V(G)|
% Such a set S must be all boundary vertices, since otherwise you could decrease our constant by removing a non-boundary vertex (keeping our numerator the same but increasing or denominator).
% But then \overline{S} would give a lower constant in our definition, a contradiction.
% Hence our constant is equal to theirs.
Note that $\partial S$ is defined differently than the boundary used in~\cite{Chung_1997}
(for us, $\partial S$ is a subset of $S$ instead of $\overline{S}$). %, but this doesn't affect the value of the Cheeger constant (simply replace the roles of $S$ and $\overline{S}$ in the definition above).
Other definitions of the Cheeger constant of a graph may count the number of boundary edges instead of boundary vertices.

For a connected graph, we have $0 < g \le 1$.
For the upper bound, consider a connected graph $G$ with more than one vertex.
If $|S| = 1$, then $\min\{|S|, |\overline S|\} = 1$ and since $G$ is connected, $|\partial S| = 1$.
Thus any connected graph has a Cheeger constant at most 1 since $g$ takes the minimum of all vertex subsets.
For the lower bound, note that if $S$ is neither empty nor all of $V(G)$, then $|\partial S|\ge 1$ when $G$ is connected.
If $G$ is disconnected then $g = 0$ since $|\partial S| = 0$ if we take $S$ to be a connected component.
Roughly speaking, a larger $g$ tells us that overall the graph has more connections; a smaller $g$ says that the graph is easier to break into pieces.
Consider the graphs in Figure~\ref{fig:cheeger_graphs} for a few examples of Cheeger constants.

\begin{figure}
    \centering
    \includegraphics[width = .6\textwidth]{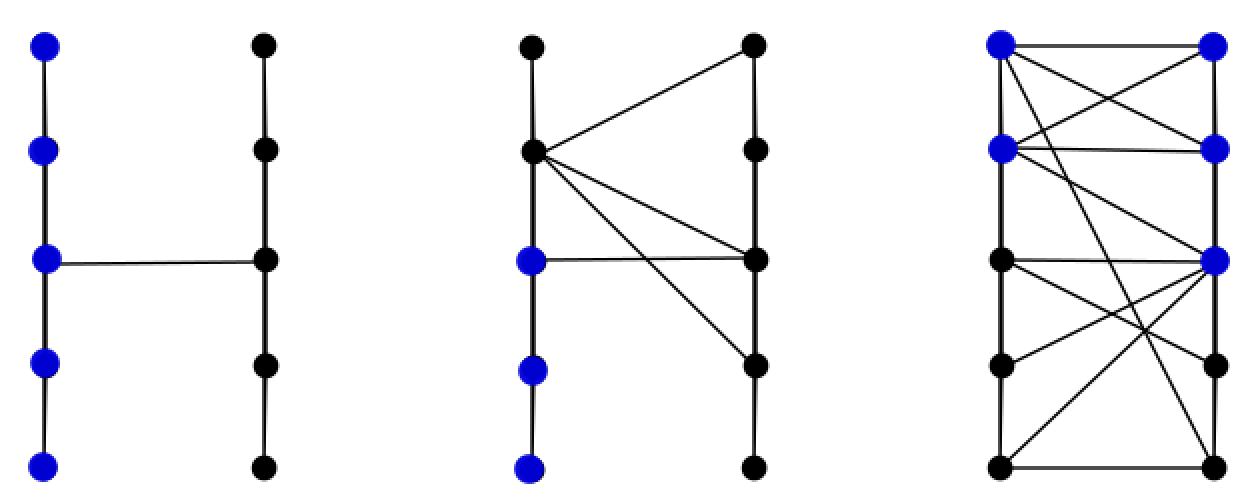}
    \caption{From left to right, the Cheeger constants are $g = \frac{1}{5}, g = \frac{1}{3}, g = \frac{2}{5}$.
    These Cheeger constants can be realized by taking the blue vertices to be the set $S$.}
    \label{fig:cheeger_graphs}
\end{figure}

We now give lower bounds on the number of lions needed to clear a graph in terms of the Cheeger constant.
Recall a graph with $\emph{polite lions}$ is one in which at each turn, at most one lion can move.
We first consider polite lions, before moving to the case of arbitrary lions.

\begin{theorem}
\label{thm:cheeger-polite}
Let $G$ be a connected graph with vertex set $V$ and with Cheeger constant $g$.
If $k \le \frac{1}{2}\lfloor\frac{|V|}{2}\rfloor g$, then $G$ cannot be cleared by $k$ polite lions.
\end{theorem}

\begin{proof}
Suppose for a contradiction that $k \le \frac{1}{2}\lfloor\frac{|V|}{2}\rfloor g$ polite lions can clear $G$.
Since the number of cleared vertices can increase by at most one in each step with polite lions, in the process of clearing there must be a time $t$ satisfying $|C(t)| = \lfloor\frac{|V|}{2}\rfloor$.
Now, by the definition of the Cheeger constant $g$ and the observation that $\min\{|C(t)|,|\overline{C(t)}|\}\le \lfloor\frac{|V|}{2}\rfloor$, we have that $\lfloor\frac{|V|}{2}\rfloor g \le |\partial C(t)|$.
Since $2k\le |\partial C(t)|$, this implies by Lemma~\ref{lem:BGGK2} that $|C(t+1)| \leq |C(t)|$.
Therefore the number of cleared vertices is always at most $\lfloor\frac{|V|}{2}\rfloor$, contradicting the clearing of $G$ with $k$ lions.
\end{proof}

\begin{theorem}
\label{thm:cheeger}
Let $G$ be a connected graph with Cheeger constant $g$.
If $k \le \frac{g|V|}{4+g}$, then $G$ cannot be cleared by $k$ lions.
\end{theorem}

\begin{proof}
Suppose for a contradiction that $k \leq \frac{g|V|}{4+g}$ lions can clear $G$.
We have that $k(2+\frac{g}{2}) \leq \frac{g|V|}{2}$, and hence $2k \leq g(\frac{|V|}{2} - \frac{k}{2})$.
This implies that for any $x$ satisfying $0 \leq x \leq \frac{k}{2}$, we have $2k \leq g(\frac{|V|}{2} - x).$
By definition of the Cheeger constant, $g \leq \frac{|\partial{S}|}{\frac{|V|}{2} - x}$, where $S$ is any subset of $V$ of size $\frac{|V|}{2} \pm x$.
Combining these two facts implies that for any $x$ satisfying $0 \leq x \leq \frac{k}{2}$, we have $2k \leq g(\frac{|V|}{2} - x) \leq  | \partial S|$,
where $S\subseteq V$ is any set of size $|S|=\frac{|V|}{2} \pm x$.
By Lemma~\ref{lem:BGGK1} the size of the cleared set can increase by at most $k$ in any step.
Therefore there is some time $t$ when the number of cleared vertices is within $\pm \frac{k}{2}$ of $\frac{|V|}{2}$ and when $|C(t+1)|>|C(t)|$.
But since $2k\le|\partial C(t)|$ at this time step $t$, Lemma~\ref{lem:BGGK2} then implies that the number of cleared vertices cannot increase in the next step, giving a contradiction.
Therefore $k$ lions do not suffice to clear the graph $G$.
\end{proof}

The advantage of using the Cheeger constant is that it gives a lower bound on the number of lions needed to clear an arbitrary graph.
That is not to say that the bound obtained by this method is near the optimal number of lions.
Rather, Theorems~\ref{thm:cheeger-polite} and~\ref{thm:cheeger} gives a weak bound for any graph, including graphs that do not have obvious symmetry that can be used to discover a better bound.

For example, suppose the Cheeger constant is $g=1$, which happens if $G$ is a complete graph.
If $g=1$, then Theorem~\ref{thm:cheeger} implies that more than $|V|/5$ lions are needed to clear the graph $G$, and Theorem~\ref{thm:cheeger-polite} implies that at least $|V|/4$ polite lions are needed.
If $g=\frac{1}{2}$, then Theorem~\ref{thm:cheeger} says that more than $|V|/9$ lions are needed, and Theorem~\ref{thm:cheeger-polite} implies that at least $|V|/8$ polite lions are needed.

\section{Conclusion and Open Questions}
\label{sec:conclusion}

The lion and contamination problem is a pursuit-evasion problem classically defined on square grid graphs.
We have explored extensions of this problem by considering restricted models of lion motion --- such as caffeinated and polite lions --- and by studying the case of triangular grid graphs.
We found an upper bound on the number of lions that that would suffice to clear a triangulated parallelogram graph under the typical lion motion and using caffeinated lions.
We also found a lower bound for the number of lions to clear a triangulated square graph using techniques from~\cite{BGGK}.
We explored a possible lower bound on the number of lions needed to clear a discretized triangle graph, and provided a sketch of a possible proof.
Lastly, we used the Cheeger constant to give a lower bound on the number of lions needed to clear contamination on an arbitrary connected graph using two types of lion motion: polite lions and arbitrary lions.
While we do not expect the Cheeger constant bound to typically be tight, particularly for graphs with extra structure or symmetry, it is quite general and applies to any connected graph.

We end with a collection of open questions, that we hope will inspire future work on the lions and contamination problem with other families of graphs.

\begin{question}{Question~1}
Let $G$ be a connected graph and let $H$ be a connected subgraph.
If $k$ lions can clear $G$, then can $k$ lions clear $H$?
Prove or find a counterexample.

This question is not a priori obvious even if $V(H)=V(G)$.
Nor is this question obvious if $H$ is an \emph{induced} subgraph of $G$, which means that $V(H)\subseteq V(G)$ and that two vertices in $H$ are connected by an edge in $H$ exactly when they are connected by an edge in $G$.
\end{question}

\begin{remark}
Note that the square grids are subgraphs of $R_n$ (with diagonals removed), so if the above question is answered in the affirmative, then our results might be able to be used to prove some results from~\cite{BGGK} --- see Theorem~\ref{thm:Dn}.
\end{remark}

\begin{question}{Question~2}
We define a sweep of a graph to be \emph{monotonic} if every vertex that is cleared never again becomes recontaminated.
If a graph $G$ admits a sweep by $k$ lions, then does it necessarily admit a monotonic sweep by $k$ lions?
And if not, then what is the smallest possible counterexample, both in terms of $k$ and in terms of the number of vertices in the graph?
For these questions, the lions are allowed to choose their starting positions.

See~\cite{lapaugh1993recontamination} for a different type of sensor motion (also in which contamination lives on the edges) in which the existence of a sweep implies the existence of a monotonic sweep.

\end{question}

\begin{question}{Question~3}
Consider the graphs $C(n,k)$ which have $n$ evenly-spaced vertices around a circle, and all edges of length at most $k$ steps around the circle.
For $k=0$, this is $n$ distinct points; for $k=1$ this is a circle graph; for $k=2$ a bunch of triangles form.
How many lions are needed to clear these graphs of contamination?
For $k=1$ it is clear that $2$ lions suffice.
For $k>1$ one can see that $2k$ lions suffice (note 4 lions are likely needed for $k=2$ and $n$ large).
What are the best lower bounds we can get on the number of lions needed?
\end{question}

\begin{question}{Question~4}
What can we say about strongly regular graphs, of type $(n,k,\lambda,\mu)$?
What bounds on  the number of lions needed can we give in terms of $n$, $k$, $\lambda$, $\mu$?
\end{question}

\begin{question}{Question~5}
Given a way to discretize a Euclidean domain into a graph, what can be said about the number of lions needed to clear the graph, perhaps as the number of vertices in discretization goes to infinity?
For domains in the plane one might expect the number of lions to scale in relationship with a length, and for domains in $\R^n$ one might expect the number of lions to scale in relationship to a $(n-1)$-dimensional volume.
In what settings are results along these lines true?
See Corollary~9.1 of~\cite{SensorSweeps} for related ideas.
What can be said about the relationship between the number of lions needed for different types of discretizations, say triangular versus square versus hexagonal triangulations of a planar domain, or different regular lattices in a three-dimensional domain, perhaps as the number of vertices goes to infinity?
\end{question}

\begin{question}{Question~6}
There are several related notions of Cheeger constants on graphs, all of which are different discretizations of the Cheeger constant of a manifold.
In Section~\ref{sec:cheeger} we consider a connection between the Cheeger constant that is sometimes referred to as the \emph{vertex isoperimetric number} $h_{in}$~\cite{bobkov2000lambda} and the number of polite lions needed.
In this Cheeger constant, the size of a boundary is counted using vertices.
A more commonly studied discretization, however, is the Cheeger constant $h_G$ of a graph~\cite{Chung_1997} where the size of a boundary is counted using edges.
Suppose we only considered \emph{monotonic} sweeps --- sweeps in which the number of cleared vertices is not allowed to ever decrease, as in Question~2 --- by caffeinated lions (lions which must move at every step).
Is the Cheeger constant $h_G$ relevant for bounding the number of lions (from below) needed for monotonic sweeps by caffeinated lions?

\end{question}

\begin{question}{Question~7}
How many lions are necessary to clear the triangular graph $P_n$ or the square graphs $S_n$ and $R_n$?
Are $n$ lions needed for all three of these graphs?
\end{question}

%%%%%%%%%%%%%%%%%%%%%%%%%%%%%%%%%%%%%%%%%%%%%%%%%%%%%%%%%%%%%%%%%%%%%%%%%%%%%%%%%%%%%%%%%%%%%%%%%%%%%%%%%%

\bibliographystyle{plain}
\bibliography{LionsAndContamination.bib}

\end{document}